\documentclass{article}
\usepackage[utf8]{inputenc}
\usepackage[T1]{fontenc}
\usepackage[english]{babel}
\usepackage[autostyle=true]{csquotes}
\usepackage{mathdots}
\usepackage{latexsym, amssymb, amsmath}
\usepackage{amsthm}
\makeatletter
\renewcommand*\env@matrix[1][*\c@MaxMatrixCols c]{%
	\hskip -\arraycolsep
	\let\@ifnextchar\new@ifnextchar
	\array{#1}}
\usepackage{url}
\usepackage{float}
\usepackage{caption}
\usepackage{subcaption}
\usepackage{color}   
\usepackage{hyperref}
\usepackage[symbol]{footmisc}
\hypersetup{
	colorlinks=true, 
	linkcolor=black,
	filecolor=magenta,      
	urlcolor=cyan,
	linktoc=all,     
}
\usepackage[ruled,vlined]{algorithm2e}
\usepackage{graphicx}
\graphicspath{ {./images/} }
\usepackage{amscd} 
\usepackage{tikz-cd}
\usepackage[all,cmtip]{xy}
\usepackage{graphicx}

\newcommand{\F}{\mathbb{F}}

\newtheorem{thm}{Theorem}[section]
\newtheorem{lemma}[thm]{Lemma}
\newtheorem{corollary}[thm]{Corollary}
\newtheorem{prop}[thm]{Proposition}

\theoremstyle{definition}
\newtheorem{definition}[thm]{Definition}

\newtheorem{ex}{Example}[thm]

\makeatletter
\def\blfootnote{\xdef\@thefnmark{}\@footnotetext}
\makeatother

\date{}

\usepackage[%
backend=bibtex,%
style=numeric-comp,%
sorting=nyt,%
hyperref,%
]{biblatex}

\AtBeginBibliography{}

\renewbibmacro{in:}{%
	\ifentrytype{article}
	{}
	{\bibstring{in}%
		\printunit{\intitlepunct}}}

\usepackage[%
backend=bibtex,%
style=numeric-comp,%
sorting=nyt,%
hyperref,%
]{biblatex}
\DeclareUnicodeCharacter{2212}{-}
\begin{document}
	\sloppy
	
	\title{Biderivations of complete Lie algebras\blfootnote{Keywords: Biderivation, complete Lie algebra, linear commuting map} \blfootnote{\textit{\textup{2020} Mathematics Subject Classification}: 17B40, 17B05, 17B20} 
	\blfootnote{The authors are supported by University of Palermo (FFR2023, UNIPA VQR24) and by the “National Group for Algebraic and Geometric Structures, and their Applications” (GNSAGA – INdAM).}} \maketitle
	\noindent
	{{Alfonso Di Bartolo}, {Gianmarco La Rosa} \\ \\
		\footnotesize{Dipartimento di Matematica e Informatica}\\
		\footnotesize{Universit\`a degli Studi di Palermo, Via Archirafi 34, 90123 Palermo, Italy}\\
		\footnotesize{alfonso.dibartolo@unipa.it}, ORCID: 0000-0001-5619-2644 \\
		\footnotesize{gianmarco.larosa@unipa.it}, ORCID: 0000-0003-1047-5993 \\

\begin{abstract}
	The authors of this article intend to present some results obtained in the study of biderivations of complete Lie algebras. Firstly they present a matricial approach to do this, which was a useful and explanatory tool not only in the study of biderivations but also in the synthesis of these results. Then they study all biderivations of a Lie algebra $L$ with $\operatorname{Z}(L)=0$ and $\operatorname{Der(L)}=\operatorname{ad(L)}$, called \emph{complete}. Moreover, as an application of the previous result, they describe all biderivations of a semisimple Lie algebra (that are complete), extending a result obtained by X.\ Tang in (\cite{tang2018}) that describes all biderivations of a complex simple Lie algebra. And thirdly, results on symmetric and skew-symmetric biderivations are also presented.
\end{abstract}

\section*{Introduction}

The first notion of biderivation dates back to 1980 when Gy.\ Maksa studied in \cite{maksa1980remark} symmetric biadditive maps with non negative diagonalization in a different context from our (Hilbert spaces), that have been revealed to be biderivations. From the 1980s the study of biderivations on rings and algebras has had a great increase \cite{Vukman1989SymmetricBO}. Moreover, as can be seen from the contributions already mentioned and those to be mentioned below, biderivations and their applications has always been linked to commuting, centralizing and relating mappings. The definition of biderivation for a Lie algebra was given in 2011 by D.\ Wang, X.\ Yu and Z.\ Chen in \cite{wangyuchen}.  A considerable number of articles appeared in the literature since then, where biderivations of Lie algebras have been studied (see \cite{wangyu2013}, \cite{hanwangxia2016}, \cite{chen2016}, \cite{changchen2019},\cite{changchenzhou2019},  \cite{changchenzhou2021} to name a few). In this paper we prove some results related to biderivations of Lie algebras in terms of matrices associated to these particular bilinear maps. This approach offers an opportunity to deepen certain aspects of this topic and use linear algebraic tools to study and expose the results obtained.\\
In the first section the reader will find some basic definitions, the definition of biderivation (in a wider context than Lie algebras) and a collection of some results about them. In the next section, we will introduce the matrix approach we will use to study biderivations on Lie algebras. In section 3 we will prove the main result of the paper which describe all biderivations of complete Lie algebras (in particular, if $L$ is semisimple). It also extends a well known result on simple Lie algebra obtained by X.\ Tang in 2018 (\cite{tang2018}). In conclusion, in the last section, we will show some results about symmetric and skew-symmetric biderivations. 

\section{Preliminaries}
Let $A$ be an associative algebra over the commutative ring $R$. A \emph{derivation} $d\colon A\rightarrow A$ is a linear map that satisfies the \emph{Leibniz identity} 
\[d(xy)=d(x) y+xd(y),\] 
for every $x,y\in A$. A similar definition can be given for a Lie algebra. Denote by $L$ a Lie algebra over a field $\F$. An example of derivation on $L$ is the linear map $\operatorname{ad}_x\colon L\rightarrow L$, with $x\in L$, that maps every $y\in L$ in $\left[x,y\right]$. These derivations are called \emph{inner}. In order to generalize the definition of derivation on Lie algebras, some researchers gave the definitions of \emph{generalized derivation}, \emph{quasiderivation}, \emph{near derivation}, etc.\ and a lot of them studied this maps in several different cases (see for example \cite{LEGER2000165} and \cite{BRESAR20083765}). Another way to generalize the definition of derivation is to considerate bilinear map instead of a linear map and require that this bilinear map is a derivation in each of its arguments. More precisely, refering to the article \cite{Brear1995OnGB} of M.\ Brešar, W.\ Martindale and C.\ Miers, a bilinear map $f\colon A\times A\rightarrow A$ is called a \emph{biderivation} of $A$ if 
\begin{align*}
	& f(xy,z)=xf(y,z)+f(x,z)y
	\\
	& f(x,yz)=f(x,y)z+yf(x,z),
\end{align*}
for all $x,y,z\in A$. Suppose now that $A$ is a noncommutative algebra and let $\left[x,y\right]=xy-yx$ be the Lie product of the elements $x,y\in A$. Then, for all $x,y\in A$ and $\lambda \in \operatorname{Z}(A)$ (the center of $A$), the map 
\[f(x,y)=\lambda\left[x,y\right]\]
is the main example of biderivation on $A$. The biderivations of this form are called \emph{inner biderivations}. In \cite{BRESAR1993342} it was proved that all biderivations on noncommutative prime rings of this type. D.\ Benkonvič in \cite{BENKOVIC20091587} proved furthermore that, under certain condition, all biderivations on triangular algebra is a sum of an extremal and an inner biderivation. This result extends that obtained by J.\ Zhang et al. in \cite{ZHANG2006225}, which stated that biderivations of nest algebras are usually inner (they showed by some examples that, in some spacial cases, there exist non-inner biderivations). Biderivations have many applications to other field (see \cite{Bresar2004centralmapsurevy} for more details). Motivating by this D.\ Wang, X.\ Zu and Z.\ Chen gave in \cite{wangyuchen} the definition of biderivation on Lie algebras to study these latter on parabolic subalgebras of Lie algebras. 

\begin{definition}\cite{wangyuchen}
	Let $L$ be a Lie algebra over a field $\F$. A bilinear map $B\colon L\times L\rightarrow L$ is called \emph{biderivation} if it satisfies 
	\begin{align}
		& B(\left[x,y\right],z)=\left[x,B(y,z)\right]+\left[B(x,z),y\right]
		\label{cond1}\\
		& B(x,\left[y,z\right])=\left[B(x,y),z\right]+\left[y,B(x,z)\right]
		\label{cond2},
	\end{align}
	for all $x,y,z\in L$.
\end{definition}

An equivalent way to define a biderivation of a Lie algebra it is the following. We note the same could be done for other algebraic structures.

\begin{definition}\label{altdefbid}
	Let $L$ be a Lie algebra over a field $\F$. A bilinear map $B\colon L\times L\rightarrow L$ is called \emph{biderivation} if the maps
	$B(x,-)$ and $B(-,y)$ are derivations of $L$, for all $x,y\in L$.
\end{definition}


\section{A matricial approach for the study of biderivations of Lie algebras}

To make the calculations more clear we use the following matricial approach. Let $L$ be a Lie algebra of dimension $n$ over a field $\F$ and let $\left\{e_1,\ldots,e_n\right\}$ be a basis of $L$. Consider now the vector space product $\operatorname{M}_n(\F)^n$ of the $n$-tuple of matrices in $\operatorname{M}_n(\F)$. A biderivation $B\colon L\times L\rightarrow L$ can be thought as an element of $\operatorname{M}_n(\F)^n$. In general, a biderivation $B$ can be written as
$$
B(x,y)=\beta_1(x,y)e_1+\cdots+\beta_n(x,y)e_n,
$$
where $\beta_1,\ldots,\beta_n\colon L\times L\rightarrow\F$ are bilinear forms. Let $B_i$ be the matrix associated with the bilinear form $\beta_i$, for every $i=1,\ldots, n$.  We denote with $\operatorname{BiDer(L)}$ the set of all biderivations of the Lie algebra $L$. Now we are ready to define the following map 
\begin{align*}
	F\colon \operatorname{BiDer(L)}&\rightarrow M_n(\F)^n \\
	B&\mapsto \left(B_1,\ldots,B_n\right).
\end{align*}
We denote with $b^k_{ij}$ the $(i,j)$-th entry of the matrix $B_k$, i.e.\ the scalar $\beta_k(e_i,e_j)$, with $i,j,k=1,\ldots,n$.\\ Let $U,V$ and $W$ be vector spaces over $\F$ and let $\operatorname{B}(U,V; W)$ be the set of all bilinear maps from $U\times V$ to $W$. By definition, we have that $\operatorname{BiDer(L)}$ is a subset of $\operatorname{B}(L,L; L)$. Now, since $\operatorname{B}(U,V; W)$ is a vector space (see Chapter 5 in \cite{schaefer1971topological}), it is natural wonder whether $\operatorname{BiDer(L)}$ is a subspace of $\operatorname{B}(L,L; L)$ or whether it is just a subset.

\begin{prop}\label{biderisvs}
	Let $L$ be a Lie algebra over a field $\F$. The set $\operatorname{BiDer(L)}$ is a subspace of $\operatorname{B}(L,L; L)$.
\end{prop}
\proof We want to show that $\operatorname{BiDer(L)}$ is a subspace of $\operatorname{B}(L,L; L)$, so we show that, for all $B_1, B_2, B\in\operatorname{BiDer(L)}$ and $\lambda\in\F$, $B_1+B_2$ and $\lambda B$ belongs to $\operatorname{BiDer(L)}$. In particular, since the biderivations are bilinear maps, $B_1+B_2$ and $\lambda B$ are bilinear maps, for every $B_1, B_2, B\in\operatorname{BiDer(L)}$ and $\lambda\in\F$. To prove that $B_1+B_2$ is a biderivation we have to show that $B_1+B_2$ satisfies ($\ref{cond1}$) and ($\ref{cond2}$). The same applies to $\lambda B$. Therefore
\begin{align*}
	(B_1+B_2)(\left[x,y\right],z)&=B_1(\left[x,y\right],z)+B_2(\left[x,y\right],z)=\\
	&=\left[x,B_1(y,z)\right]+\left[B_1(x,z),y\right]+\left[x,B_2(y,z)\right]+\left[B_2(x,z),y\right]=\\
	&=\left[x,B_1(y,z)+B_2(y,z)\right]+\left[B_1(x,z)+B_2(x,z),y\right]=\\
	&=\left[x,(B_1+B_2)(y,z)\right]+\left[(B_1+B_2)(x,z),y\right],
\end{align*}
for all $B_1, B_2\operatorname{BiDer(L)}$ and $x,y,z\in L$;

\begin{align*}
	(\lambda B)(\left[x,y\right],z)&=\lambda B(\left[x,y\right],z)=\lambda\left(\left[x,B(y,z)\right]+\left[B(x,z),y\right]\right)=\\
	&=\lambda\left[x,B(y,z)\right]+\lambda\left[B(x,z),y\right]=\\
	&=\left[x,(\lambda B)(y,z)\right]+\left[(\lambda B)(x,z),y\right],
\end{align*}
for all $\lambda\in\F$ and $B\in\operatorname{BiDer(L)}$.
In a similar manner we could show that $B_1+B_2$ and $\lambda B$ satisfy the condition $(\ref{cond2})$. 
\endproof

\begin{prop}
	The map
	\begin{align*}
		F\colon \operatorname{BiDer(L)}&\rightarrow \operatorname{M}_n(\F)^n \\
		B&\mapsto \left(B_1,\ldots,B_n\right).
	\end{align*}
	is a monomorphism of vector spaces.
\end{prop}
\proof The map $F$ is linear since the biderivations are bilinear maps. Furthermore $F(B)=\left(0,0,\ldots,0\right)$ if and only if $B=0$, since $B_1=B_2=\cdots=B_n=0_n\in \operatorname{M}_n(\F)$.
\endproof


In the following proposition we will rewrite the the conditions ($\ref{cond1}$) and ($\ref{cond2}$) in terms of bilinear forms associated with a biderivation of a Lie algebras and its structure constants.

\begin{prop}
	Let $L$ be a Lie algebra over a field $\F$ and $\left\{e_1,\ldots,e_n\right\}$ a basis of $L$. Let $\left\{c_{ij}^k\right\}$ be the structure constants of L, that is, $\left[e_,e_j\right]=\sum_{k=1}^{n}c_{ij}^ke_k$ for every $i,j,k=1,\ldots,n$. Then $B\colon L\times L\rightarrow L$ is a biderivation of $L$ if and only if  
	\[
	\sum_{t=1}^{n} c_{jk}^tb_{it}^r=\sum_{t=1}^{n}\left(c_{tk}^rb_{ij}^t+c_{jt}^rb_{ik}^t\right) \text{ and }  \sum_{t=1}^{n} c_{ij}^tb_{tk}^r=\sum_{t=1}^{n}\left(c_{tj}^rb_{ik}^t+c_{jt}^rb_{jk}^t\right),
	\]
	for every $i,j,k,r=1,\ldots, n$.
\end{prop}
\proof From $(\ref{cond1})$ we have
\[ B(\left[e_i,e_j\right],e_k)=\left[B(e_i,e_k),e_j\right]+\left[e_i,B(e_j,e_k)\right],
\] for every $i,k,k=1,\ldots,n$. Then
\[ B(\left[e_i,e_j\right],e_k)=B(\sum_{t=1}^{n}c_{ij}^te_t,e_k)=\sum_{t=1}^{n}c_{ij}^tB(e_t, e_k)=\sum_{r=1}^{n}\sum_{t=1}^{n}c_{ij}^tb_{tk}^re_r;
\]
on the other hand, 
\begin{gather*}
	\left[B(e_i,e_k),e_j\right]+\left[e_i,B(e_j,e_k)\right]=\\
	=[\sum_{t=1}^{n}b_{ik}^te_t,e_j]+[e_i,\sum_{t=1}^{n}b_{jk}^te_t]=\sum_{t=1}^{n}b_{ik}^t\left[e_t,e_j\right]+\sum_{t=1}^{n}b_{jk}^t\left[e_i,e_t\right]=\\
	=\sum_{r=1}^{n}\sum_{t=1}^{n}c_{tj}^rb_{ik}^te_r+\sum_{r=1}^{n}\sum_{t=1}^{n}c_{it}^rb_{jk}^te_r=\sum_{r=1}^{n}(\sum_{t=1}^{n}c_{tj}^rb_{ik}^t+c_{it}^rb_{jk}^t)e_r.
\end{gather*}
Since $\left\{e_i,\ldots,e_n\right\}$ is a basis of $L$, by comparing both equations we obtain the first equation. With similar computations we obtain the second equation.
\endproof

In light of this result one can asks if, in matricial terms, a relation between biderivations and derivations exists and, if so, what kind of relation there is. The next result proves that there is an affirmative answer in this sense and, besides, describes this relation. Bur first we introduce some notations. Let $(B_1,\ldots, B_n)$ the $n$-tuple of matrices associated to a biderivation $B$ of a $n$-dimensional Lie algebra $L$ respect to a fixed basis. We denote with $\left(B_i\right)^k$ the $k$-th row of the matrix $B_i$ and with $\left(B_i\right)_j$ the $j$-th column of $\left(B_i\right)$, for all $i=1,\ldots,n$. 

\begin{prop}
	Let $B\colon L\times L\rightarrow L$ be a bilinear map of a $n$- dimensional Lie algebra $L$ and let $B_1,\ldots,B_n \in  \operatorname{M}_n(\F)$ the matrices associated to $B$. 
	$B$ is a biderivation $L$ if and only if, for every $i=1,\ldots, n$, 
	\[
	\begin{pmatrix}
		\left(B_1\right)^i\\
		\left(B_2\right)^i\\
		\vdots\\
		\left(B_n\right)^i\\
	\end{pmatrix}
	\,\,\text{ and }\,\,
	\begin{pmatrix}
		\left(B_1\right)_i & \left(B_2\right)_i & \cdots & \left(B_n\right)_i
	\end{pmatrix}\]
	are matrices associated to two derivations of $L$.

\end{prop}
\proof 
The "if" directions follows directly by the definition of derivation. The  other direction is not trivial. Let $\delta^x:=B(x,-)$ be the linear function that maps every $y\in L$ in $B(x,y)$, for every $x\in L$. Since $B$ is a biderivation, $\delta^x$ is a derivation of $L$ for every $x\in L$. Let $\mathcal{B}=\left\{e_1,\ldots,e_n\right\}$ be a basis of $L$. Thus, for every $e_i\in\left\{e_1,\ldots,e_n\right\}$, we have
\begin{gather*}
	\delta^{e_j}(e_i)=B(e_j,e_i)=\beta_1(e_j,e_i)e_1+\cdots+\beta_n(e_j,e_i)e_n=\\
	=b_{ji}^1e_1+\cdots+b_{ji}^ne_n,
\end{gather*} 
where $\beta_1,\ldots,\beta_n$ are the bilinear forms associated to the biderivation $B$ respect to the basis $\mathcal{B}$. 
The matrix associated to the derivation $\delta^{e_j}$ is
\[\begin{pmatrix}
	b_{ji}^1 & b_{j2}^1 & \cdots & b_{jn}^1 \\
	b_{ji}^2 & b_{j2}^2 & \cdots & b_{jn}^2 \\
	\vdots & \vdots & \ddots &\vdots \\
	b_{ji}^n & b_{j2}^n & \cdots & b_{jn}^n \\
\end{pmatrix}.\]
It is clear that the $i$-th row of the matrix above is the $j$-th row of $B_i$, the $i$-th matrix associated to the biderivation $B$. If we define the linear function $\delta_x:=B(-,x)$, the same arguments may be used to prove that the $i$-th row of the matrix associated to the derivation $\delta_{e_j}$ is the $j$-th column of the matrix $B_i$.
\endproof

Remind that the set $\operatorname{BiDer(L)}$ is a vector space, as we have seen before. In addition, every biderivation $B$ has an image in $\operatorname{M}_n(\F)^n$ via the linear map $\varphi$. It is clear that not all $n$-tuple of matrices in $\operatorname{M}_n(\F)$ correspond to a biderivation of $L$. For example:

\begin{ex}
	Let $L=L_{2,2}=\langle e_1, e_2\rangle_\F$ be a Lie algebra of dimension 2 over $\F$ such that $\left[e_1,e_2\right]=e_1$. The pair of matrices 
	\[\left(\begin{pmatrix}
		0 & 0 \\ 0 & 1
	\end{pmatrix}, \begin{pmatrix}
		1 & 0 \\ 0 & 0
	\end{pmatrix}\right)
	\]
	is not image of a biderivation $B$ because $B(e_1,e_1)=e_2$, $B(e_2,e_2)=e_1$ and
	\begin{gather*}
		B(\left[e_1,e_2\right]e_1)=B(e_1,e_1)=e_2 \\
		\left[e_1,B(e_2,e_1)\right]+\left[B(e_1,e_1),e_2\right]=\left[e_1,0\right]+\left[e_2,e_2\right]=0.
	\end{gather*}
\end{ex}
This is just an example of how hard could be to define a Lie bracket on $\operatorname{BiDer(L)}$ such that $\varphi$ is a Lie monomorphism. The Lie bracket of $\operatorname{M}_n(\F)$ is not closed in general, in fact. In the next proposition we show how a Lie bracket could be defined on the set of biderivations.
\begin{prop}
	If $F(\operatorname{BiDer(L)})\leq \operatorname{M}_n(\F)^n$ is a Lie subalgebra, then $\left\{\operatorname{BiDer(L)},\left\{-,-\right\}\right\}$ is a Lie algebra, where \[\left\{-,-\right\}=F^{-1}\circ\left[-,-\right]_{\operatorname{M}_n(\F)^n}.\]
\end{prop}
\proof 
Since $J:=F(\operatorname{BiDer(L)})$ is a Lie subalgebra of $\operatorname{M}_n(\F)^n$, then $J$ is closed under the Lie bracket. In addition, $F\colon\operatorname{BiDer}\rightarrow J$ is a linear isomorphism. Thus, for every $A,B\in\operatorname{BiDer(L)}$, we have
\[
\left\{A,B\right\}=F^{-1}(\left[A,B\right])=F^{-1}\left(\left[A_1,B_1\right]_{\operatorname{M}_n(\F)},\ldots,\left[A_n,B_n\right]_{\operatorname{M}_n(\F)}\right).
\]
Since $\left[-,-\right]_{\operatorname{M}_n(\F)^n}$ is a Lie bracket and $F^{-1}$ is a linear isomorphism, the bilinear map $\left\{-,-\right\}$ defined above is a Lie bracket on $\operatorname{BiDer(L)}$. 
\endproof

We would to conclude this first section with a couple of result regarding biderivations on two-step nilpotent Lie algebras. Remind that the \emph{lower central series} defined recursively as the series $L^1=L'$ and $L^k=\left[L,L^{k-1}\right]$, for $k\geq2$. $L$ is \emph{nilpotent} if exists $k\geq1$ such that $L^{k-1}\neq0$ and $L^k=0$. A Lie algebra $L$ is \emph{two-step nilpotent} if $k=2$. For such Lie algebras the ideal commutator ideal $L'$ is contained in the center of $L$.  On the other hand, the condition $L'\subseteq \operatorname{Z}(L)$ implies that $L$ is a two-step nilpotent Lie algebra.

\begin{prop}
	Let $L$ be a two-step nilpotent Lie algebra over a field $\F$ and $B$ a biderivation of $L$. Then, for every $x\in L$ and $z\in L'$, $B(x,z), B(z,x)\in L'\subseteq Z(L)$.
\end{prop}
\begin{proof}
	Let $\left\{e_1,\ldots,e_n,\right\}$ be a basis of $L$. For every $z\in L'$,  $z=\sum_{i,j=1}^{n}\alpha_{ij}\left[e_i,e_j\right]$ for some $\alpha_{ij}\in\F$. Thus we have
	\begin{gather*}
		B(x,z)=B\left(x,\sum_{i,j=1,i<j}^{n}\alpha_{ij}\left[e_i,e_j\right]\right)=\\
		=\sum_{i,j=1}^{n}\alpha_{ij}B(x,\left[e_i,e_j\right])=\sum_{i,j=1}^{n}\alpha_{ij}(\left[B(x,e_i), e_j\right]+\left[e_i,B(x,e_j)\right])\in L'.
	\end{gather*}

	With similar computations we obtain $B(z,x)$ and these results prove the statement.
\end{proof}
\begin{corollary}
	Let $L$ be a two-step nilpotent Lie algebra and $B$ a biderivation of $L$. Then, for every $z,z'\in L'$, $B(z,z')=0$.
\end{corollary}

\section{Biderivations of complete Lie algebras}

X.\ Tang in \cite{tang2018} proved that all biderivations of a complex simple Lie algebra are inner. We will extend this result. We begin this section with a result that makes it easier the study of biderivations of complete Lie algebras (in particular, if the Lie algebra is semisimple). 
\\
In 1962 N.\ Jacobson gave in \cite{jacobson1979lie} the definition of \emph{complete} Lie algebra, that is a Lie algebra $L$ with $\operatorname{Z}(L)=0$ and $\operatorname{Der(L)}=\operatorname{ad(L)}$. The next result makes easier the study of biderivations of this class of Lie algebras.

\begin{prop}\label{propbiderss}
	Let $L$ be a complete Lie algebra over a field $\mathbb{F}$. $B$ is a biderivation of $L$ if and only if exist two linear maps \hbox{$\varphi,\psi\in\operatorname{End(L)}$} such that, for every $x,y\in L$,
	\[B(x,y)=\left[\varphi(x),y\right]=\left[x,\psi(y)\right].\]
\end{prop}
\proof
The "if" direction is trivial to prove. To prove the other direction we recall that, by Definition \ref{altdefbid}, $B(x,-)$ and $B(-,x)$ are derivations of $L$, for every $x\in L$. Since $L$ is complete, all derivations of $L$ are inner. Thus, for every $x\in L$, there exist $u,v\in L$ such that $B(x,-)=\operatorname{ad}_u(-)=\left[u,-\right]$ and $B(-,x)=\operatorname{ad}_v(-)=\left[v,-\right]$. So we can define two maps $\varphi,\psi:L\rightarrow L$ in the following way
\[
\varphi\colon x\mapsto u\,\,\text{ and }\,\,\psi\colon x\mapsto -v.
\]
Now we have to prove that $\varphi$ and $\psi$ are linear. By definition of $\varphi$ we have \[
\varphi(x+y)(t)=\operatorname{ad}_{\varphi(x+y)}(t)=\left[\varphi(x+y),t\right], 
\]
for every $t\in L$. Moreover, 
\[\varphi(x)(t)+\varphi(y)(t)=\operatorname{ad}_{\varphi(x)}(t)+\operatorname{ad}_{\varphi(y)}(t)=\left[\varphi(x),t\right]+\left[\varphi(y),t\right]=\left[\varphi(x)+\varphi(y),t\right].
\]
Since $B$ is a bilinear map, for every $t\in L$ we have $\left[\varphi(x+y),t\right]=\left[\varphi(x)+\varphi(y),t\right]$ and this implies that $\varphi(x+y)-\varphi(x)-\varphi(y)\in Z(L)$. The center of $L$ is zero because $L$ is complete, then $\varphi(x+y)=\varphi(x)+\varphi(y)$ for all $x,y\in L$.
By similar computations we can affirm that $\psi$ is also linear. To conclude the proof we observe that $\left[\varphi(x),y\right]$ and $\left[x,\psi(y)\right]$ are biderivations of $L$ because the Lie bracket is skew-symmetric and verifies the Jacoby identity.
\endproof

With the same assumptions made in the last proposition we prove the following result.

\begin{lemma}\label{lembiderss}
	If $B(x,y)=\left[\varphi(x),y\right]=\left[x,\psi(y)\right]$ is a biderivation of $L$ and $\varphi=\lambda\operatorname{id}_L$ for some $\lambda\in\mathbb{F}$, then $\varphi=\psi$.
\end{lemma}
\proof For any $x,y\in L$ we have
\[B(x,y)=\left[\varphi(x),y\right]=\left[\lambda x,y\right]=\lambda\left[x,y\right].\] 
On the other hand $B(x,y)=\left[x,\psi(y)\right]$, then $\left[x,\lambda y-\psi(y)\right]=0$, for all $x,y\in L$. This implies that $\lambda y-\psi(y)$ belongs to the center of $L$, that is trivial because $L$ is complete and this conlcude the proof.\endproof
It is well known that all derivations of a simple Lie algebra are inner (see \cite{erdmann2006introduction}) and this happens also for semisimple Lie algebras (see \cite{humphreys2012introduction}). Biderivations of complex simple Lie algebras are studied by X.\ Tang in \cite{tang2018} where he proved the following result.

\begin{thm}\cite{tang2018}\label{thmtang}
	Suppose that $L$ is a finite-dimensional complex simple Lie algebra. Then $B$ is a biderivation of $L$ if and only if it is inner, i.e.\ there is a complex number $\lambda$ such that
	\[B(x,y)=\lambda\left[x,y\right], \, \forall x,y\in L.\]
\end{thm}

One can asks if something like that happens to semisimple Lie algebras and the answer is affirmative. We would to remind that every semisimple Lie algebra has trivial center and all derivations on it are inner. Thus every semisimple Lie algebra is complete. While it is straightforward to demonstrate that all semisimple Lie algebras are complete, proving that the converse (that all complete Lie algebras are semisimple) is false, is not as obvious. 
E. Angelopoulos constructed in \cite{angelopoulos} a class of \emph{sympathetic} Lie algebras, i.e.\ complete Lie algebras with $\left[L,L\right]=L$, which are not semisimple. Notably, there exists a counter-example within this class of minimal dimension, namely a Lie algebra of dimension 35, whose Levi subalgebra is isomorphic to $\mathfrak{sl}(2)$. Now we are ready to prove the main result of this section. From now on we indicate with $P'$ the transpose matrix of a matrix $P$.

\begin{thm}\label{thmbiderss}
	Let $L=L_1\oplus \cdots \oplus L_t$ be an $n$-dimensional complex semisimple Lie algebra, where $L_i$ is a complex simple Lie algebra with $\dim_\mathbb{C} L_i=n_i$ for every $i=1,\ldots, t$ and $n_1+\cdots+n_t=n$. A bilinear map $B\colon L\times L\rightarrow L$ is a biderivation of $L$ if and only if exist $\lambda_1,\ldots,\lambda_t\in\mathbb{C}$ such that
	\[B(x_1+\cdots+x_t, y_1+\cdots+y_t)=\lambda_1\left[x_1,y_1\right]+\cdots+\lambda_t\left[x_t,y_t\right],\]
	with $x_i,y_i\in L_i$, $i=1,\ldots,t$.
\end{thm}
\proof For every $i=1,\ldots,t$ let $\mathcal{B}_i=\left\{e_{i_1},\ldots,e_{i_{n_i}}\right\}$ be a basis of $L_i$ and let $\left\{(c_i)^k_{lm}\right\}_{l,m,k=1,\ldots, n_i}$ be the structure constants of $L_i$, i.e.\ $\left[e_{i_l},e_{i_m}\right]=\sum_{k=1}^{n_i}(c_i)^k_{lm}e_{i_k}$ for every $e_{i_l},e_{i_m}\in\mathcal{B}_i$. Then the structure matrices of $L_i$ forms an $n_i$-tupla of $n_i\times n_i$ matrices and let this tupla be $\left(C_{i_1}\ldots, C_{i_{n_i}}\right)$, where $C_{i_j}=(c_i)^j_{lm}$ with $j=1,\ldots, n_i$. Thus the structures matrices of $L$ are $\gamma_{ij}$, with $i=1,\ldots,t$ and, for any $i$, $j=1,\ldots, n_i$, where
\[\gamma_{ij}=\begin{pmatrix}
	0 & 0 & \cdots & \cdots & 0 \\
	& \ddots &  & & \\
	\vdots & & C_{i_j} & & \vdots \\
	&  &  & \ddots & \\
	0 & \cdots & \cdots & 0 & 0 
\end{pmatrix}\]
is the diagonal block matrix where the $i$-th diagonal block is the matrix $C_{i_j}\in \operatorname{M}_{n_i}(\mathbb{C})$. By Proposition $\ref{propbiderss}$ there exist $\varphi,\psi\in\operatorname{End(L)}$ such that 
\begin{equation}\label{eqthm1}
	B(x,y)=\left[\varphi(x),y\right]=\left[x,\psi(y)\right]
\end{equation} for every $x,y\in L$. Thus there exist two matrices $P,Q\in \operatorname{M}_n(\mathbb{C})$ associated with, respectively, $\varphi$ and $\psi$ respect to the basis $\mathcal{B}=\mathcal{B}_1\cup\cdots\cup\mathcal{B}_t$ of $L$. With this assumptions Equations (\ref{eqthm1}) hold if and only if $P'\gamma_{ij}=\gamma_{ij}Q$, for any $i=1,\ldots,t$ and $j=1,\ldots,n_i$.
The matrices $P=(P_{ij})_{i,j=1,\ldots,t}$ e $Q=(Q_{ij})_{i,j=1,\ldots,t}$ are block matrices, where $P_{ij},Q_{ij}\in \operatorname{M}_{n_i\times n_j}(\mathbb{C})$. The condition $P'\gamma_{ij}=\gamma_{ij}Q$ implies that $P'_{ik}C_{i_j}=0=C_{i_j}Q_{ik}$, for every $k\in\left\{1,\ldots,t\right\}\setminus\left\{i\right\}$, because
\[
P'\gamma_{ij}=\begin{pmatrix}
	0 & \cdots & 0 & P'_{i1}C_{i_j} & 0 & \cdots & 0 \\ 
	0 & \cdots & 0 & P'_{i2}C_{i_j} & 0 & \cdots & 0 \\ 
	\vdots & \ddots & \vdots & \vdots & \vdots & \ddots & 0 \\ 
	0 & \cdots & 0 & P'_{it}C_{i_j} & 0 & \cdots & 0 \\
\end{pmatrix}
\] and
\[
\gamma_{ij}Q=\begin{pmatrix}
	0 & 0 & \cdots  & 0 \\
	\vdots & \vdots & \ddots & \vdots \\
	0 & 0 & \cdots  & 0 \\
	C_{i_j}Q_{i1} & C_{i_j}Q_{i2} & \cdots & C_{i_j}Q_{it} \\
	0 & 0 & \cdots  & 0 \\
	\vdots & \vdots & \ddots & \vdots \\
	0 & 0 & \cdots  & 0 \\
\end{pmatrix}.
\] 
Then, for every $x\in L_k, y\in L_i$, we have
\begin{gather*}
	x'P'_{ik}C_{i_j}y=0\Rightarrow (P_{ik}x)'C_{i_j}y=0,\,\,\forall j=1,\ldots, n_i \Rightarrow \\
	\Rightarrow\left((P_{ik}x)'C_{i_1}y,\ldots, (P_{ik}x)'C_{i_{n_i}}y \right)=(0,\ldots,0)\cong0_{L_i}.
\end{gather*}
The matrix $P_{ik}\in \operatorname{M}_{n_i\times n_k}(\mathbb{C})$ is the matrix associated to a linear map $f\colon L_k\rightarrow L_i$ respect to the basis $\mathcal{B}_k$ of $L_k$ and $\mathcal{B}_i$ of $L_i$. Thus $\left[f(x),y\right]_{L_i}=0_{L_i}$, i.e.\ $f(x)\in \operatorname{Z}(L_i)$. The centers $\operatorname{Z}(L_i)=\left\{0_{L_i}\right\}$ since $L_i$ is a simple Lie algebra, for every $i=1,\ldots,t$, and then $P_{ik}=0\in \operatorname{M}_{n_i\times n_k}(\mathbb{C})$. By similar computation we obtain $Q_{ik}=0\in \operatorname{M}_{n_i\times n_k}(\mathbb{C})$. Then $P$ and $Q$ are respectively the diagonal block matrices
\[\begin{pmatrix}
	P_{11} & 0 & \cdots & 0 \\
	0 & P_{22} &  & 0 \\
	\vdots &  & \ddots & \vdots \\
	0 & 0 & \cdots & P_{tt} 
\end{pmatrix}\quad\text{ and }\quad
\begin{pmatrix}
	Q_{11} & 0 & \cdots & 0 \\
	0 & Q_{22} &  & 0 \\
	\vdots &  & \ddots & \vdots \\
	0 & 0 & \cdots & Q_{tt} 
\end{pmatrix}.\]
The linear maps $\varphi_i$ and $\psi_i$, whose associated matrices are respectively $P_{ii}$ and $Q_{ii}$ respect to the basis $\mathcal{B}_i$ of $L_i$, map every element $x_i\in L_i$ in $L_i$, for any $i=1,\ldots,t$. Finally, for every $x=x_1+\cdots+x_t, y=y_1+\cdots+y_t\in L=L_1\oplus \cdots \oplus L_t$ we have
\begin{gather*}
	B(x,y)=B(x_1+\cdots+x_t, y_1+\cdots+y_t)=\left[\varphi_1(x_1),y_1\right]_{L_1}+\cdots+\left[\varphi_t(x_t),y_t\right]_{L_t}=\\
	=\left[x_1,\psi_1(y_1)\right]_{L_1}+\cdots+\left[x_t,\psi_t(y_t)\right]_{L_t}.
\end{gather*}
By Proposition \ref{propbiderss} and since $L_i$ is a simple Lie algebra, exist $\lambda_i,\mu_i\in\mathbb{C}$ such that $\left[\varphi_i(x_i),y_i\right]_{L_i}=\lambda_i\left[x_i,y_i\right]_{L_i}$ and $\left[x_i,\psi_i(y_i)\right]_{L_i}=\mu_i\left[x_i,y_i\right]_{L_i}$, for every $i=1,\ldots,t$. We can conclude the proof by Lemma \ref{lembiderss} and say that $\lambda_i=\mu_i$,  for every $i=1,\ldots,t$.  \endproof

In the last part of this section we want to show how to decompose the vector space $\operatorname{BiDer(L)}$ in a direct sum of its subspaces, when $L$ is a complete Lie algebra. In other words, we will show that is always possible to decompose the vector space $\operatorname{BiDer(L)}$ of all biderivations of $L$ in a direct sum of two simplier vector spaces. The space $\operatorname{BiDer(L)}$ is isomorphic to the vector space
\[\{\varphi\in\operatorname{End(L)}\,|\,\exists\psi\in\operatorname{End(L)} \text{ such that } \left[\varphi(x),y\right]=\left[x,\psi(y)\right], \forall x,y\in L\}\]
that is isomorphic to
\[V=\{ P\in \operatorname{M}_n(\mathbb{C})\,|\,\exists Q\in \operatorname{M}_n(\mathbb{C})\text{ such that } PA_i=A_iQ, \forall i=1,\ldots, n \},\]
where $\dim_\mathbb{C}L=n$ and $(A_1,\ldots,A_n)$ are the structure matrices of $L$.

\begin{thm}\label{thmdec}
	Let $L$ be a complete Lie algebra over the complex field $\mathbb{C}$, with $\dim_\mathbb{C}L=n$, $(A_1,\ldots,A_n)$ the structure matrices of $L$ respect to a fixed basis $\mathcal{B}$ of $L$ and let $V$ be the vector space defined above. Then $V=V^+\oplus V^-$, where
	\begin{gather*}
		V^+= \left\{ P\in \operatorname{M}_n(\mathbb{C})\,|\, (PA_i)'=(PA_i)\right\}\\
		V^-= \left\{ P\in \operatorname{M}_n(\mathbb{C})\,|\, (PA_i)'=-(PA_i)\right\}
	\end{gather*}
\end{thm}
\proof Remind that all matrices $A_i$ are skew-symmetric. Then $V^+\oplus V^-\subseteq V$ because, for all $P_+\in V^+, P_-\in V^-$ and $A_i\in\left\{A_1,\ldots,A_n\right\}$, we have 
\[
P_+A_i=(P_+A_i)'=A_i'P_+'=-A_iP_+'=A_i(-P_+'),
\] and this proves that $P_+\in V$. On the other hand,
\[
P_-A_i=-(P_-A_i)'=-A_i'(P_-)'=A_i(P_-'),
\]  
for every $A_i\in\left\{A_1,\ldots,A_n\right\}$, and this proves that $P_-\in V$.  Now, to prove that $V\subseteq V^+\oplus V^-$ we consider $P\in V$ and $Q\in \operatorname{M}_n\mathbb(C)$ such that $PA_i=A_iQ$. Then we have
\begin{gather*}
	\left(\left(P+Q'\right)A_i\right)'=\left(PA_i+Q'A_i\right)'=\left(A_iQ+Q'A_i\right)'=\\
	=-Q'A_i-A_iQ=-PA_i-Q'A_i=-\left(P+Q'\right)A_i
\end{gather*} and
\begin{gather*}
	\left(\left(P-Q'\right)A_i\right)'=\left(PA_i-Q'A_i\right)'=\left(A_iQ-Q'A_i\right)'=\\
	=-Q'A_i+A_iQ=\left(-Q'A_i+PA_i\right)=\left(P-Q'\right)A_i,
\end{gather*}
for every $A_i\in\left\{A_1,\ldots,A_n\right\}$. Then we showed that $P+Q'\in V^-$ and $P-Q'\in V^+$.
Thus the matrix $P=\frac{1}{2}\left(P+Q'\right)+\frac{1}{2}(P-Q')$ belongs to $V^+\oplus V^-$ and this proves the second inclusion.\endproof

By Theorem \ref{thmbiderss} it follows that the matrix $P$ associated to the endomorphism $\varphi$ of the semisimple Lie algebra $L$ (that is equal to Q, the matrix associated to $\psi$) is direct sum of scalar matrices $\lambda_i I_{n_i}$, for $i=1,\ldots,t$, then $P$ is a scalar matrix and  $(PA_i)'=A_i'P'=-A_iP'=-A_iP$, for all structure matrices $A_i\in\left\{A_1,\ldots,A_n\right\}$. Thus the subspace $V^+$ is empty and if $L=L_1\oplus \cdots \oplus L_t$, where $L_i$ is a complex simple Lie algebra with $\dim_\mathbb{C} L_i=n_i$ for every $i=1,\ldots, t$ and $n_1+\cdots+n_t=n$, $V^-$ is isomorphic to $\mathbb{C}^t$. We summarize this facts in the following proposition.

\begin{prop}
	Let $L$ be a semisimple Lie algebra over the complex field $\mathbb{C}$ like above, with $\dim_\mathbb{C}L=n$, $(A_1,\ldots,A_n)$ the structure matrices of $L$ respect to a fixed basis $\mathcal{B}$ of $L$ and $V=V^+\oplus V^-$ the vector space defined in Theorem \ref{thmdec}. Then $\operatorname{BiDer(L)}\cong V^-\cong\mathbb{C}^t$.
\end{prop}

\section{Symmetric and skew-symmetric biderivations of complete Lie algebras}

There are several examples of papers in which biderivations (more precisely skew-symmetric biderivations) are determined in terms of linear commuting maps (see \cite{bresarzhao2018},\cite{chen2016},\cite{chengwangsunzhang2017},\cite{hanwangxia2016}, \cite{wangyu2013}). In this section we show how simply are symmetric and skew-symmetric biderivations of a complete Lie algebra.\\
A biderivation $B\colon L\times L\rightarrow L$ is called \emph{symmetric} (resp. skew-symmetric) if $B(x,y)=B(y,x)$ (resp. $B(x,y)=-B(y,x)$), for all $x,y\in L$. If $B$ is a symmetric biderivation the bilinear forms $\beta_1,\beta_2,\ldots,\beta_n$ associated with $B$ are symmetric and then the matrices $B_1,B_2,\ldots,B_n$ are symmetric. Obviously, the same reasoning applies to a skew-symmetric biderivation. It is equivalent to say that $\varphi$ maps symmetric (resp. skew-symmetric) biderivations into $n-$tuples of symmetric (resp. skew-symmetric) matrices. 
In general, every biderivation $B\colon L\times L\rightarrow L$ of a Lie algebra $L$ can be written as $B=\frac{1}{2}B^++\frac{1}{2}B^-$, where $B^+$ and $B^-$ are respectively the bilinear maps from $L\times L$ to $L$ defined as
\[
B^+\colon (x,y)\mapsto B(x,y)+B(y,x)\,\,\text{  and  }\,\, B^-\colon (x,y)\mapsto B(x,y)-B(y,x),
\]
for all $x,y\in L$. We note that, since $\operatorname{BiDer(L)}$ is a vector space, $B^+$ and $B^-$ are biderivations of $L$. With this assumptions we prove the following result.
\begin{prop}
	Let $B$ be a biderivations of a complete Lie algebra $L$, with
	\begin{gather*}
		B(x,y)=\left[\varphi(x),y\right]=\left[x,\psi(y)\right]
	\end{gather*}
	for some $\varphi,\psi\in\operatorname{End(L)}$ and for all $x,y\in L$. 
	\begin{itemize}
		\item If $B$ is symmetric, then $\varphi=-\psi$.\\
		\item If $B$ is skew-symmetric, then $\varphi=\psi$.
	\end{itemize}
\end{prop}
\begin{proof}{\ }\\
	\begin{itemize}
		\item Since $B$ is a symmetric biderivation of $L$ we have $B(x,y)=B(y,x)$, for every $x,y\in L$. Then
		$B(y,x)=\left[\varphi(y),x\right]$ and $B(x,y)=\left[\varphi(x),y\right]=\left[x,\psi(y)\right])=-\left[\psi(y),x\right]$. By comparing these two results we
		obtain $\left[\varphi(y)+\psi(y),x\right]=0$ and this implies that $\varphi(y)+\psi(y)\in \operatorname{Z}(L)$. The Lie algebra L is centerless because
		it is complete and this allows us to conclude that $\varphi=-\psi$.
		\item  If we start from the equality $B(y,x)=-B(x,y)$, with the same arguments and similar computations we obtain $\varphi=\psi$.
	\end{itemize}
\end{proof}
We recall here that a linear map $f:L\rightarrow L$ is called \emph{commuting} if $\left[x,f(x)\right]=0$, for every $x\in L$. If the characteristic $char(\F)\neq2$, then a linear map $f\colon L\rightarrow L$ is commuting if and only if $\left[f(x),y\right]=\left[x,f(y)\right]$ for all $x,y\in L$. In a similar way $f$ is called \emph{skew-commuting} if and only if $\left[f(x),y\right]=-\left[x,f(y)\right]$ for all $x,y\in L$. To be
more precise the definition of commuting linear maps can be given for a wider class of algebraic structures, for example rings. After these definitions and the last proposition, it is fairly straightforward to prove the following results in which
are described symmetric and skew-symmetric biderivations of a complete Lie algebra $L$.
\begin{corollary}
	Let $L$ be a complete Lie algebra. $B\colon L\rightarrow L$ is a symmetric biderivation of $L$ if and only if there exists
	a unique skew-commuting linear map $\varphi\in\operatorname{End(L)}$ such that $B(x,y)=\left[\varphi(x),y\right]$, for any $x,y\in L$.
\end{corollary} 
\begin{corollary}
	Let $L$ be a complete Lie algebra. $B\colon L\rightarrow L$ is a skew-symmetric biderivation of $L$ if and only if there exists
	a unique commuting linear map $\varphi\in\operatorname{End(L)}$ such that $B(x,y)=\left[\varphi(x),y\right]$, for any $x,y\in L$.
\end{corollary}

\section*{Declarations} 
Not applicable. There is no Competing Interest.

\printbibliography


	
	
	

\end{document}